\documentclass[oneside,english]{amsart}
\usepackage[T1]{fontenc}
\usepackage[latin9]{inputenc}
\usepackage{amstext}
\usepackage{amsthm}
\usepackage{amssymb}

\makeatletter
\numberwithin{equation}{section}
\numberwithin{figure}{section}
\theoremstyle{plain}
\newtheorem{thm}{\protect\theoremname}[section]
  \theoremstyle{plain}
  \newtheorem{lem}[thm]{\protect\lemmaname}
  \theoremstyle{plain}
  \newtheorem{prop}[thm]{\protect\propositionname}
  \theoremstyle{remark}
  \newtheorem{rem}[thm]{\protect\remarkname}

\usepackage{alex}
\usepackage{tikz}
 \usetikzlibrary{matrix}

\makeatother

\usepackage{babel}
  \providecommand{\lemmaname}{Lemma}
  \providecommand{\propositionname}{Proposition}
  \providecommand{\remarkname}{Remark}
\providecommand{\theoremname}{Theorem}

\begin{document}

\title[Commutators of trace zero]{Commutators of trace zero matrices over principal ideal rings}

\author{Alexander Stasinski}
\begin{abstract}
We prove that for every trace zero square matrix $A$ of size at least
$3$ over a principal ideal ring $R$, there exist trace zero matrices
$X,Y$ over $R$ such that $XY-YX=A$. Moreover, we show that $X$
can be taken to be regular mod every maximal ideal of $R$. This strengthens
our earlier result that $A$ is a commutator of two matrices (not
necessarily of trace zero), and in addition, the present proof is
simpler than the earlier one.

Shalev has conjectured an analogous statement for group commutators
in $\SL_{n}$ over $p$-adic integers. We prove Shalev's conjecture
for $n=2$.
\end{abstract}

\address{Department of Mathematical Sciences, Durham University, South Rd,
Durham, DH1 3LE, UK}

\email{alexander.stasinski@durham.ac.uk}

\maketitle

\section{Introduction}

Let $R$ be a principal ideal ring, which we will always take to be
commutative with identity (e.g., $R$ could be a field). We let $\gl_{n}(R)$
denote the Lie algebra of $n\times n$ matrices over $R$ with Lie
bracket $[X,Y]=XY-YX$, and $\sl_{n}(R)$ the sub Lie algebra of trace
zero matrices. In case $R=K$ is a field, a theorem of Albert and
Muckenhoupt \cite{Albert-Muckenhoupt} says that every $A\in\sl_{n}(K)$
is a commutator in $\gl_{n}(K)$, that is, there exist $X,Y\in\gl_{n}(K)$
such that $[X,Y]=A$. To go beyond the field case requires new ideas
and the first major step was taken by Laffey and Reams \cite{Laffey-Reams}
who proved the analogous result for $R=\Z$, solving a problem posed
by Vaserstein \cite[Section~5]{Vaserstein/87}. Whether every element
in $\sl_{n}(R)$ is a commutator in $\gl_{n}(R)$ for a PIR $R$,
was an open problem going back implicitly at least to Lissner \cite{Lissner},
and was settled in the affirmative in \cite{commutatorsPID-2016}.

In light of the above results, a natural question is whether $X$
and $Y$ can be taken in $\sl_{n}(R)$, rather than just $\gl_{n}(R)$.
When $R=K$ is a field, it is known by work of Thompson \cite[Theorems~1-4]{Thompson-tracezero}
that any $A\in\sl_{n}(K)$ can be written as $A=[X,Y]$ for some $X,Y\in\sl_{n}(K)$,
except when $\chara K=2$ and $n=2$. A generalisation of Thompson's
result, allowing $X$ and $Y$ to lie in an arbitrary hyperplane in
$\gl_{n}(K)$ (but assuming $n>2$ and $|K|>3$), was recently obtained
by de~Seguins Pazzis \cite{SeguinsPazzis}. On the other hand, it
does not seem possible to modify our proof in \cite{commutatorsPID-2016}
to yield the stronger assertion that every $A\in\sl_{n}(R)$, with
$n\geq3$, is a commutator of matrices in $\sl_{n}(R)$, even in the
case where $R$ is a field.

The main result of the present paper is that for any principal ideal
domain (henceforth PID) $R$ and $A\in\sl_{n}(R)$, with $n\geq3$,
there exist $X,Y\in\sl_{n}(R)$ such that $A=[X,Y]$. It is also easy
to see that when $2$ is invertible in $R$, the same conclusion holds
for $A\in\sl_{2}(R)$. Moreover, it follows from our proof that $X$
can be chosen to be regular mod every maximal ideal of $R$ (this
was stated as an open problem in \cite{commutatorsPID-2016}). Our
proof is significantly simpler than the proof of the main result in
\cite{commutatorsPID-2016}, and the new idea is to consider the matrices\newlength\mtxrowsep   \setlength\mtxrowsep{0.01\arraystretch\baselineskip} \newlength\mtxcolsep   \setlength\mtxcolsep{0.3\arraycolsep}
$$X(\bfx,a)=
\begin{tikzpicture}[baseline=(current bounding box.center), 
column sep=\mtxcolsep,     row sep=\mtxrowsep,
every left delimiter/.style={xshift=9pt},     every right delimiter/.style={xshift=-5pt}]
\matrix (m) [matrix of math nodes,nodes in empty cells,right delimiter=),left delimiter=(]
{0 & 0 & 0 &  & 0\\
 x_{1} & 0 & 1 &  & \\
 &   & 0 &  & 0\\
 & 0 &  &  & 1\\
x_{n-1} & a & 0 &  & 0 \\ };

\draw[loosely dotted,thick] (m-3-3)-- (m-5-5); 
\draw[loosely dotted,thick] (m-2-1)-- (m-5-1); 
\draw[loosely dotted,thick] (m-1-3)-- (m-3-5); 
\draw[loosely dotted,thick] (m-1-3)-- (m-1-5); 
\draw[loosely dotted,thick] (m-1-5)-- (m-3-5); 
\draw[loosely dotted,thick] (m-3-3)-- (m-5-3); 
\draw[loosely dotted,thick] (m-2-3)-- (m-4-5);
\draw[loosely dotted,thick] (m-2-2)-- (m-4-2);
\draw[loosely dotted,thick] (m-5-3)-- (m-5-5);
\end{tikzpicture}
\in \sl_n(R),$$
where $\bfx=(x_{1},\dots,x_{n-1})^{\mathsf{{T}}}\in R^{n-1}$ and
$a\in R$; see Section~\ref{sec:The-matricesX}. These matrices have
some remarkable properties which let us carry through the proof. More
precisely, we show that for a given non-scalar $A\in\sl_{n}(R)$ in
Laffey\textendash Reams form (see \cite[Theorem~5.6]{commutatorsPID-2016}),
we can find $\bfx$ and $a$ such that 
\[
\tr(X(\bfx,a)^{r}A)=0,\quad\text{for }r=1,\dots,n-1,
\]
and at the same time ensure that $X(\bfx,a)$ mod $\mfp$ is regular
in $\gl_{n}(R/\mfp)$, for every maximal ideal $\mfp$ of $R$, as
well as regular in $\sl_{n}(R/\mfp)$, for any $\mfp$ for which $A$
is non-scalar mod $\mfp$. We note that the condition on the vanishing
of traces above is rather delicate, given that we also want $X(\bfx,a)$
to have the above regularity property and trace zero, and depends
on the existence of a solution of a system of polynomial equations
over $R$, which in most cases is hopelessly complicated. Nevertheless,
for the matrices $X(\bfx,a)$ the system of equations becomes atypically
simple, and we are able to show that a solution exists. We then use
the well known local-global principle for systems of linear equations
over rings, applied to the system defined by $[X(\bfx,a),Y]=A$, $Y\in\sl_{n}(R)$.
Working over the localisation $R_{\mfp}$ at a maximal ideal $\mfp$
of $R$, we use a variant of the criterion of Laffey and Reams (see
Section~\ref{sec:LF-criterion}, Proposition~\ref{prop:Criterion})
to show that the system has a solution if $A$ is non-scalar mod $\mfp$.
Here we use that $A$ mod $\mfp$ is not merely regular in $\gl_{n}(R/\mfp)$
but also regular in $\sl_{n}(R/\mfp)$. The existence of a solution
over $R_{\mfp}$ when $\mfp$ is such that $A$ mod $\mfp$ is scalar
is more subtle and requires a separate argument. The existence of
a local solution for every maximal ideal $\mfp$ then implies the
existence of global solution, and since any non-scalar matrix is $\GL_{n}(R)$-conjugate
to one in Laffey-Reams form, our main result follows (the case when
$A$ is scalar requires a separate discussion, but is easy).

Once the main result has been established for a PID, it is easy to
deduce it for an arbitrary principal ideal ring (not necessary and
integral domain).

In \cite[Conjecture~1.3]{Shalev-conjecture} and \cite[Conjecture~8.10]{Shalev-ErdosCentennial}
Shalev conjectured that if either $n\geq3$ or $n=2$ and $p>3$,
then any element in $\SL_{n}(\Z_{p})$ is a commutator of elements
in $\SL_{n}(\Z_{p})$ (here $\Z_{p}$ denotes the $p$-adic integers).
The main result of the present paper, in the case where $R=\Z_{p}$,
is therefore a Lie algebra analogue of this conjecture. In \cite{AGKS-SLnZp}
some results towards this conjecture were obtained. In Section~\ref{sec:Shalevs-conj}
we give a proof of Shalev's conjecture for $n=2$. 

We end this introduction with a word on notation. A ring (without
further specification) will mean a commutative ring with identity.
Throughout, we will use $1_{n}$ to denote the identity matrix in
$\gl_{n}(S)$, where $S$ is a ring with identity. If $X\in\gl_{n}(S)$,
$S[X]$ will denote the unital $S$-algebra generated by $X$.

\section{\label{sec:LF-criterion}The criterion of Laffey and Reams }

In this section, $K$ denotes an arbitrary field. We will prove an
analogue of the Laffey\textendash Reams criterion (see \cite[Section~3]{Laffey-Reams}
and \cite[Proposition~3.3]{commutatorsPID-2016}) for a matrix in
$\sl_{n}(R)$, $R$ a \emph{local} PID, to be a commutator of matrices
in $\sl_{n}(R)$. This criterion plays a key role in our proof of
the main theorem. 

We need a couple of remarks about regular elements in $\sl_{n}(K)$.
It is well known that an element $X\in\gl_{n}(K)$ is regular if and
only if 
\[
C_{\gl_{n}(K)}(X)=K[X],
\]
that is, if and only if the centraliser of $X$ in $\gl_{n}(K)$ has
dimension $n$. In this situation, we will say that $X$ is $\gl_{n}(K)$-\emph{regular.}
Similarly, if $X\in\sl_{n}(K)$ we define $X$ to be $\sl_{n}(K)$-\emph{regular}
if 
\[
\dim C_{\sl_{n}(K)}(X)=n-1.
\]
For $X\in\sl_{n}(K)$ it may happen that $X$ is $\gl_{n}(K)$-regular
but not $\sl_{n}(K)$-regular: take for example $\left(\begin{smallmatrix}0 & 0\\
1 & 0
\end{smallmatrix}\right)\in\sl_{n}(\F_{2})$.

The following result describes the precise relationship between the
properties $\sl_{n}$-regular and $\gl_{n}$-regular over a field.
\begin{lem}
\label{lem:gl_n-sl_n-regular}Let $X\in\sl_{n}(K)$. Then the following
holds:
\begin{enumerate}
\item If $X$ is $\sl_{n}(K)$-regular, then $X$ is $\gl_{n}(K)$-regular.
\item $X$ is $\sl_{n}(K)$-regular if and only if it is $\gl_{n}(K)$-regular
and $\tr(K[X])\neq0$.
\item If $\chara K$ does not divide $n$, then an element $X$ is $\sl_{n}(K)$-regular
if and only if it is $\gl_{n}(K)$-regular.
\end{enumerate}
\end{lem}
\begin{proof}
For the first part, note that $C_{\sl_{n}(K)}(X)$ is either equal
to $C_{\gl_{n}(K)}(X)$ or is a hypersurface in $C_{\gl_{n}(K)}(X)$,
so $C_{\sl_{n}(K)}(X)$ has codimension at most one in $C_{\gl_{n}(K)}(X)$.
Thus $X$ being $\sl_{n}(K)$-regular implies that $\dim C_{\gl_{n}(K)}(X)\leq n$.
But it is well-known that the dimension of a centraliser in $\gl_{n}(K)$
is always at least $n$, so $X$ is $\gl_{n}(K)$-regular. 

For the second part, first note that $C_{\sl_{n}(K)}(X)$ is the kernel
of the trace map $\tr:C_{\gl_{n}(K)}(X)\rightarrow K$. Now, if $X$
is $\sl_{n}(K)$-regular, then by the previous part, $X$ is $\gl_{n}(K)$-regular,
so $C_{\gl_{n}(K)}(X)=K[X]$. Thus $\dim C_{\sl_{n}(K)}(X)=n-1$ implies
that this trace map is surjective, that is, that $\tr(K[X])\neq0$.
Conversely, if $X$ is $\gl_{n}(K)$-regular and $\tr(K[X])\neq0$,
then then $\dim C_{\gl_{n}(K)}(X)=n$ and $\tr:C_{\gl_{n}(K)}(X)\rightarrow K$
is surjective, so the kernel has dimension $n-1$.

Finally, when $\chara K$ does not divide $n$ and $X$ is $\gl_{n}(K)$-regular,
then $\tr(1_{n})=n\neq0$, so the previous part implies that $X$
is $\sl_{n}(K)$-regular.
\end{proof}
\begin{prop}
\label{prop:LF-criterion-fields}Let $X\in\sl_{n}(K)$ be $\sl_{n}(K)$-regular
and let $A\in\sl_{n}(K)$. Then $A=[X,Y]$ for some $Y\in\sl_{n}(K)$
if and only if $\Tr(X^{r}A)=0$ for all $r=1,\dots,n-1$.
\end{prop}
\begin{proof}
Since $X$ is $\gl_{n}(K)$-regular by Lemma~\ref{lem:gl_n-sl_n-regular},
the set $\{1_{n},X,\dots,X^{n-1}\}$ is linearly independent over
$K$, so the subspace
\[
V=\{B\in\sl_{n}(K)\mid\Tr(X^{r}B)=0\text{ for }r=1,\dots,n-1\}
\]
has dimension $n^{2}-n$. The kernel of the linear map $\sl_{n}(K)\rightarrow\sl_{n}(K)$,
$Y\mapsto[X,Y]$ is equal to the centraliser $C_{\sl_{n}(K)}(X)$,
which has dimension $n-1$ since $X$ is $\sl_{n}(K)$-regular. Thus
the image $[X,\sl_{n}(K)]$ of the map $Y\mapsto[X,Y]$ has dimension
$n^{2}-n$. But if $A\in[X,\sl_{n}(K)]$, there exists a $Y\in\sl_{n}(K)$
such that for every $r=1,\dots,n-1$ we have 
\[
\Tr(X^{r}A)=\Tr(X^{r}(XY-YX))=\Tr(X^{r+1}Y)-\Tr(X^{r}YX)=0.
\]
Thus $[X,\sl_{n}(K)]\subseteq V$. Since $\dim V=\dim[X,\sl_{n}(K)]$
we conclude that $V=[X,\sl_{n}(K)]$. 
\end{proof}
If $S$ is a ring, $I\subseteq S$ an ideal and $X\in\gl_{n}(S)$,
we denote by $X_{I}$ the image of $X$ under the canonical map $\gl_{n}(S)\rightarrow\gl_{n}(S/I)$.
\begin{lem}
\label{lem:surjective}Let $S$ be a local commutative ring (with
identity) with maximal ideal $\mfm$. Let $X\in\sl_{n}(S)$ be such
that $X_{\mfm}$ is $\sl_{n}(S/\mfm)$-regular. Then the canonical
map
\[
C_{\sl_{n}(S)}(X)\longrightarrow C_{\sl_{n}(S/\mfm)}(X_{\mfm})
\]
is surjective.
\end{lem}
\begin{proof}
As $C_{\sl_{n}(S/\mfm)}(X_{\mfm})$ has dimension $n-1$ and is the
kernel of the trace map $\tr:C_{\gl_{n}(S/\mfm)}(X_{\mfm})\rightarrow S/\mfm$,
this map must be surjective. Thus, there exists an $a\in C_{\gl_{n}(S/\mfm)}(X_{\mfm})$
such that $\tr(a)=1$. Since $X_{\mfm}$ is $\sl_{n}(S/\mfm)$-regular,
it is also $\gl_{n}(S/\mfm)$-regular, so
\[
C_{\gl_{n}(S/\mfm)}(X_{\mfm})=(S/\mfm)[X_{\mfp}].
\]
Let $\hat{a}\in S[X]$ be any lift of $a$. Then $\tr(\hat{a})\in1+\mfm$,
so $\tr(\hat{a})$ is a unit since $S$ is a local ring. Since $S[X]\subseteq C_{\gl_{n}(S)}(X)$,
we conclude that the trace map $\tr:C_{\gl_{n}(S)}(X)\rightarrow S$
is surjective. The surjectivity of the map in the lemma is now a formal
consequence. Indeed, let $b\in C_{\sl_{n}(S/\mfm)}(X_{\mfm})\subseteq(S/\mfm)[X]$,
and take a lift $\hat{b}\in S[X]$ of $b$. Then $\tr(\hat{b})\in\mfm$,
so there exists a $c\in\mfm C_{\gl_{n}(S)}(X)$ such that $\tr(c)=\tr(\hat{b})$
(namely let $c=\tr(\hat{b})c'$, where $\tr(c')=1$). Then $\hat{b}-c\in C_{\sl_{n}(S)}(X)$,
and the image of $\hat{b}-c$ in $C_{\sl_{n}(S/\mfm)}(X_{\mfm})$
is $b$.
\end{proof}
The following result is a local version of the criterion of Laffey
and Reams (\cite[Proposition~3.3]{commutatorsPID-2016}), with the
difference that we need $X_{\mfp}$ to be $\sl_{n}(R/\mfp)$-regular
to ensure that $Y\in\sl_{n}(R)$ rather than just in $\gl_{n}(R)$.
\begin{prop}
\label{prop:Criterion}Assume that $R$ is a local PID with maximal
ideal $\mfp$, let $A\in\sl_{n}(R)$ and let $X\in\sl_{n}(R)$ be
such that $X_{\mfp}$ is $\sl_{n}(R/\mfp)$-regular. Then $A=[X,Y]$
for some $Y\in\sl_{n}(R)$ if and only if $\Tr(X^{r}A)=0$ for $r=1,\dots,n-1$.
\end{prop}
\begin{proof}
Clearly the condition $\Tr(X^{r}A)=0$ for all $r\geq1$ is necessary
for $A$ to be of the form $[X,Y]$ with $Y\in\sl_{n}(R)$. Conversely,
suppose that $\Tr(X^{r}A)=0$ for $r=1,\dots,n-1$. We claim that
$X$ is $\sl_{n}(F)$-regular, considered as an element of $\sl_{n}(F)$.
Indeed, by \cite[Proposition~2.6]{commutatorsPID-2016} $X$ is $\gl_{n}(F)$-regular,
and since $X_{\mfp}$ is $\sl_{n}(R/\mfp)$-regular, there exists
an element $a\in R[X]$ such that $\tr(a)\neq0$. Thus $\tr(F[X])\neq0$,
and so $X$ is $\sl_{n}(F)$-regular by Lemma~\ref{lem:gl_n-sl_n-regular}.

Now, by Proposition~\ref{prop:LF-criterion-fields} we have $A=[X,M]$
for some $M\in\sl_{n}(F)$. Let $p$ be a generator of $\mfp$. Then
there exists a non-negative integer $m$ such that $p^{m}M\in\sl_{n}(R)$,
and we have $[X,p^{m}M]=p^{m}[X,M]=p^{m}A$. Choose $m$ to be minimal
with respect to the property that $[X,C]=p^{m}A$ for some $C\in\sl_{n}(R)$.
Assume that $m>0$. Then $[X_{\mfp},C_{\mfp}]=0$, so $X_{\mfp}$
commutes with $C_{\mfp}$. Since $X_{\mfp}$ is $\sl_{n}(R/\mfp)$-regular,
there exists a $\hat{C}\in C_{\sl_{n}(R)}(X)$ such that $\hat{C}_{\mfp}=C_{\mfp}$,
by Lemma~\ref{lem:surjective}. Thus $C=\hat{C}+pD$, for some $D\in\sl_{n}(R)$,
so 
\[
[X,C]=[X,pD]=p[X,D]=p^{m}A.
\]
Cancelling a factor of $p$, we obtain a contradiction to the minimality
of $m$. Thus $m=0$, and the result is proved.
\end{proof}

\section{\label{sec:The-matricesX}The matrices $X(\bfx,a)$}

Let $S$ be a ring (commutative with identity), $n\geq3$, $\bfx=(x_{1},\dots,x_{n-1})^{\mathsf{{T}}}\in S^{n-1}$
and $a\in S$. The key to our main result is to consider the following
matrices:$$X(\bfx,a)=
\begin{tikzpicture}[baseline=(current bounding box.center), 
column sep=\mtxcolsep,     row sep=\mtxrowsep,
every left delimiter/.style={xshift=9pt},     every right delimiter/.style={xshift=-5pt}]
\matrix (m) [matrix of math nodes,nodes in empty cells,right delimiter=),left delimiter=(]
{0 & 0 & 0 &  & 0\\
 x_{1} & 0 & 1 &  & \\
 &   & 0 &  & 0\\
 & 0 &  &  & 1\\
x_{n-1} & a & 0 &  & 0 \\ };

\draw[loosely dotted,thick] (m-3-3)-- (m-5-5); 
\draw[loosely dotted,thick] (m-2-1)-- (m-5-1); 
\draw[loosely dotted,thick] (m-1-3)-- (m-3-5); 
\draw[loosely dotted,thick] (m-1-3)-- (m-1-5); 
\draw[loosely dotted,thick] (m-1-5)-- (m-3-5); 
\draw[loosely dotted,thick] (m-3-3)-- (m-5-3); 
\draw[loosely dotted,thick] (m-2-3)-- (m-4-5);
\draw[loosely dotted,thick] (m-2-2)-- (m-4-2);
\draw[loosely dotted,thick] (m-5-3)-- (m-5-5);
\end{tikzpicture}
\in \sl_n(S),$$
that is, $X(\bfx,a)=(m_{ij})$, where$$\begin{cases}
m_{i,i+1}=1 & \text{for }i=2,\dots,n-1,\\ m_{i1}=x_{i-1} & \text{for }i=2,\dots,n-2,\\
m_{n,2}=a & \\
m_{ij}=0 & \text{otherwise}. \end{cases}$$We can write $X(\bfx,a)$ in block form as
\[
X(\bfx,a)=\begin{pmatrix}0 & \overline{0}\\
\bfx & P
\end{pmatrix},
\]
where $\overline{0}=(0,\dots,0)$ is a $1\times n$ matrix and $P=(p_{ij})$,
$1\leq i,j\leq n-1$, where $p_{i,i+1}=1$ for $i=1,\dots,n-2$, $p_{n-1,1}=a$
and $p_{ij}=0$ otherwise. Thus, $P$ is the (row-wise) companion
matrix of the polynomial $x^{n-1}-a$.
\begin{lem}
\label{lem:Pr-1-trace}Let $P\in\sl_{n-1}(S)$ be as above, and let
$\bfy=(y_{1},\dots,y_{n-1})^{\mathsf{{T}}}\in S^{n-1}$. Then, for
any $z\in S$, and $r=1,\dots,n-1$, we have
\[
\tr(P^{r-1}\bfy(z,0,\dots,0))=zy_{r}.
\]
\end{lem}
\begin{proof}
Write $P^{r-1}=(p_{ij}^{(r-1)})$, for $1\leq i,j\leq n-1$. Since
each column in $\bfy(z,0,\dots,0)$, except for the first one, is
zero, we have 
\[
\tr(P^{r-1}\bfy(z,0,\dots,0))=(p_{11}^{(r-1)},p_{12}^{(r-1)},\dots,p_{1,n-1}^{(r-1)})z\bfy.
\]
Since $P$ is a companion matrix, there exists a $v\in S^{n-1}$ such
that $\{v,Pv,\dots,P^{n-2}v\}$ is an $S$-basis for $S^{n-1}$ and
$P$ is the matrix of the linear map defined by $P$ with respect
to this basis. Thus, for each $r=1,\dots,n-1$, the first row of $P^{r-1}$
is $(p_{11}^{(r-1)},p_{12}^{(r-1)},\dots,p_{1,n-1}^{(r-1)})$, where
$p_{1r}^{(r-1)}=1$ and all other $p_{1j}=0$. Hence
\[
(p_{11}^{(r-1)},p_{12}^{(r-1)},\dots,p_{1,n-1}^{(r-1)})z\bfy=zy_{r},
\]
and the lemma follows.
\end{proof}
\begin{lem}
\label{lem:X_n(a,b)-powers}For $r=1,\dots,n-1$ we have
\[
X(\bfx,a)^{r}=\begin{pmatrix}0 & \overline{0}\\
P^{r-1}\bfx & P^{r}
\end{pmatrix},
\]
 In particular, $\tr(X(\bfx,a)^{r})=0$ for $r=1,\dots,n-2$, and
$\tr(X(\bfx,a)^{n-1})=(n-1)a$.
\end{lem}
\begin{proof}
The expression for $X(\bfx,a)^{r}$ follows easily, using block-multiplication
of matrices. The assertion about the trace of $X(\bfx,a)^{r}$ for
$r=1,\dots,n-2$ follows from a simple induction argument, proving
that for each $r=1,\dots,n-2$, we have $P^{r}=(p_{ij}^{(r)})$, where
$p_{i,i+r}^{(r)}=1$ for $i=1,\dots,n-1-r$ and $p_{n-1-r+j,j}^{(r)}=a$
for $j=1,\dots,r$, and $p_{ij}^{(r)}=0$ otherwise. Finally, the
relation $\tr(X(\bfx,a)^{n-1})=(n-1)a$ follows from the fact that
the characteristic polynomial of $P$ is $x^{n-1}-a$.
\end{proof}
\begin{lem}
\label{lem:X_n(a,b)-regular}Let $K$ be a field, $x_{1},\dots,x_{n-1}\in K^{n-1}$
and $a\in K$. If either $x_{n-1}\neq0$ or $a\neq0$, then $X(\bfx,a)$
is $\gl_{n}(K)$-regular. If $a\neq0$, then $X(\bfx,a)$ is $\sl_{n}(K)$-regular.
\end{lem}
\begin{proof}
For simplicity, write $X=X(\bfx,a)$. We will show that if $x_{n-1}\neq0$
or $a\neq0$, then $X$ is $\gl_{n}(K)$-regular, by showing that
$\{1_{n},X,\dots,X^{n-1}\}$ is linearly independent. Lemma~\ref{lem:X_n(a,b)-powers}
implies that $\{1_{n},X,\dots,X^{n-2}\}$ is linearly independent
because $P$ is regular, so $\{1_{n-1},P,\dots,P^{n-2}\}$ is linearly
independent. Moreover, by Lemma~\ref{lem:X_n(a,b)-powers} and its
proof, we have 
\[
X^{n-1}=\begin{pmatrix}0 & \overline{0}\\
P^{n-2}\bfx & a1_{n-1}
\end{pmatrix},\qquad\text{where}\qquad P^{n-2}\bfx=\begin{pmatrix}x_{n-1}\\
ax_{1}\\
\vdots\\
ax_{n-2}
\end{pmatrix}.
\]
Thus, since $P^{i}$ has zero diagonal for all $r=1,\dots,n-2$ (see
the proof of Lemma~\ref{lem:X_n(a,b)-powers}), we conclude that
$X^{n-1}$ is not a linear combination of $1_{n},X,\dots,X^{n-2}$
if $a\neq0$. On the other hand, if $a=0$ and $x_{n-1}\neq0$, then
$X^{n-1}$ is the matrix whose $(2,1)$-entry is $x_{n-1}$ and all
other entries are zero. Since each matrix in $\{1_{n},X,\dots,X{}^{n-2}\}$
has a non-zero $(i,j)$-entry for some $(i,j)\neq(2,1)$, we conclude
that $X^{n-1}$ is not a linear combination of $1_{n},X,\dots,X{}^{n-2}$
if $a=0$ and $x_{n-1}\neq0$.

Suppose now that $a\neq0$; then $X$ is $\gl_{n}(K)$-regular. If
$\chara K\nmid n$, Lemma~\ref{lem:gl_n-sl_n-regular} implies that
$X$ is $\sl_{n}(K)$-regular. On the other hand, if $\chara K\mid n$,
then 
\[
\tr(X{}^{n-1})=(n-1)a=-a,
\]
by Lemma~\ref{lem:X_n(a,b)-powers}, so $\tr(K[X])\neq0$ and Lemma~\ref{lem:gl_n-sl_n-regular}
implies that $X$ is $\sl_{n}(K)$-regular.
\end{proof}

\section{The field case}

In this section we give a proof of our main result in the case where
$R=K$ is a field. We give a separate proof in this case, as it is
simpler than for a general PID. The result over a field was first
proved by Thompson \cite{Thompson-tracezero}, who also showed that,
apart for some small exceptions, one of the matrices $X$ can in fact
be taken to be nilpotent. We give a new proof of Thompson's result,
but instead of showing that $X$ can be chosen to be nilpotent, we
show that it can be taken to be $\gl_{n}(K)$-regular (and often $\sl_{n}(K)$-regular).

First let $n=2$. For $x,y,z,s,t,u\in K$ we have
\[
\left[\begin{pmatrix}x & y\\
z & -x
\end{pmatrix},\begin{pmatrix}s & t\\
u & -s
\end{pmatrix}\right]=\begin{pmatrix}uy-tz & 2(tx-sy)\\
2(sz-ux) & tz-uy
\end{pmatrix}.
\]
Thus, if $\chara K=2$, a matrix in $\sl_{2}(K)$ is of the form $[X,Y]$
for $X,Y\in\sl_{2}(K)$ if and only if it is scalar. On the other
hand, if $\chara K\geq3$ and $a,b,c\in K$, then 
\[
\begin{pmatrix}a & b\\
c & -a
\end{pmatrix}=\begin{cases}
\left[\left(\begin{smallmatrix}0 & 1\\
-\frac{c}{b} & 0
\end{smallmatrix}\right),\left(\begin{smallmatrix}-\frac{b}{2} & 0\\
a & \frac{b}{2}
\end{smallmatrix}\right)\right] & \text{if }b\neq0,\vspace{5pt}\\
\left[\left(\begin{smallmatrix}0 & 0\\
1 & 0
\end{smallmatrix}\right),\left(\begin{smallmatrix}\frac{c}{2} & -a\\
0 & -\frac{c}{2}
\end{smallmatrix}\right)\right] & \text{if }b=0.
\end{cases}
\]
Note that all of the matrices involved in the above commutators are
$\gl_{n}(K)$-regular. 
\begin{lem}
\label{lem:Identity-commutator}Let $S$ be a ring (commutative with
identity) such that $n=1+\dots+1=0$ in $S$. Then, for every $\lambda\in S$
there exist $X,Y\in\sl_{n}(S)$ such that $X$ is $\gl_{n}(S)$-regular
and $[X,Y]=\lambda1_{n}$.
\end{lem}
\begin{proof}
Take $X=(x_{ij})$, where $x_{i,i+1}=1$ for $i=1,\dots,n-1$ and
$x_{ij}=0$ otherwise, and $Y=(y_{ij})$, where $y_{j+1,j}=j$, for
$j=1,\dots,n-1$ and $y_{ij}=0$ otherwise. Then $X$ is a companion
matrix, hence regular as an element of $\gl_{n}(S)$. A direct computation
shows that $[X,Y]=1_{n}$, because $-(n-1)=1$ in $S$, and thus $[X,\lambda Y]=\lambda1_{n}$.
\end{proof}
\begin{rem}
\label{rem:not-sl_n-reg}If $S=K$ is a field, Lemma~\ref{lem:Identity-commutator}
does not hold if $X$ is required to be $\sl_{n}(K)$-regular; in
fact, the $X$ in the lemma is necessarily not $\sl_{n}(K)$-regular,
unless $\lambda=0$. The author was alerted to the following simple
argument by a referee: Suppose that $[X,Y]=\lambda1_{n}$ where $\lambda\neq0$
and $X$ is $\gl_{n}(K)$-regular. Then $\tr(X^{i}\lambda1_{n})=\lambda\tr(X^{i})=0$,
hence $\tr(X^{i})=0$, for all $i=0,\dots,n-1$. Thus $X$ is not
$\sl_{n}(K)$-regular, by Lemma~\ref{lem:gl_n-sl_n-regular}.
\end{rem}
\begin{thm}
\label{thm:Main-fields}Let $K$ be a field and $A\in\sl_{n}(K)$,
with $n\geq3$. Then there exist $X,Y\in\sl_{n}(K)$  such that $[X,Y]=A$.
Moreover, if $A$ is scalar, $X$ can be chosen to be $\gl_{n}(K)$-regular
and if $A$ is non-scalar, $X$ can be chosen to be $\sl_{n}(K)$-regular.
\end{thm}
\begin{proof}
Assume first that $A$ is scalar. By \cite[Proposition~4.1]{commutatorsPID-2016},
there exist $X,Y\in\gl_{n}(K)$ with $X$ $\gl_{n}(K)$-regular, such
that $[X,Y]=A$. If $\chara K$ does not divide $n$, then $[X-\frac{\tr(X)}{n}1_{n},Y-\frac{\tr(Y)}{n}1_{n}]=[X,Y]=A$,
where $X-\frac{\tr(X)}{n}1_{n},Y-\frac{\tr(Y)}{n}1_{n}\in\sl_{n}(K)$
and $X-\frac{\tr(X)}{n}1_{n}$ is $\gl_{n}(K)$-regular. On the other
hand, if $\chara K$ divides $n$, then the desired assertion follows
from Lemma~\ref{lem:Identity-commutator}. 

Assume now that $A$ is not scalar and let $A=(a_{ij})$. Then the
rational canonical form implies that after a possible $\GL_{n}(K)$-conjugation,
we can assume that $a_{11}=0$, $a_{12}=1$ and $a_{ij}=0$ whenever
$j\geq i+2$. We will show that $x_{1},\dots,x_{n-1}\in K$ can be
chosen such that $\tr(X(\bfx,1)^{r}A)=0$ for each $r=1,\dots,n-1$.
By Lemma~\ref{lem:X_n(a,b)-powers} we have 
\[
X(\bfx,1)^{r}=\begin{pmatrix}0 & \overline{0}\\
P^{r-1}\bfx & P^{r}
\end{pmatrix},
\]
where $P=(p_{ij})$, $1\leq i,j,\leq n-1$ is such that $p_{i,i+1}=1$
for $i=1,\dots,n-2$, $p_{n-1,1}=1$ and $p_{ij}=0$ otherwise. Writing
$A$ in block-form, we have
\[
A=\begin{pmatrix}0 & (1,0,\dots,0)\\
\bfa & Q
\end{pmatrix},
\]
where $\bfa$ is an $n\times1$ matrix and $Q\in\gl_{n-1}(K)$. Thus
\[
X(\bfx,1)^{r}A=\begin{pmatrix}0 & \overline{0}\\
P^{r}\bfa & Q'
\end{pmatrix},
\]
where $Q'=P^{r-1}\bfx(1,0,\dots,0)+P^{r}Q$. Thus, by Lemma~\ref{lem:Pr-1-trace},
\[
\tr(X(\bfx,1)^{r}A)=\tr(Q')=x_{r}+\tr(P^{r}Q),
\]
for each $r=1,\dots,n-1$. Put $x_{r}=-\tr(P^{r}Q)$, so that $\tr(X(\bfx,1)^{r}A)=0$,
for $r=1,\dots,n-1$. By Lemma~\ref{lem:X_n(a,b)-regular} $X(\bfx,1)$
is $\sl_{n}(K)$-regular, so Proposition~\ref{prop:LF-criterion-fields}
implies that there exists a $Y\in\sl_{n}(K)$ such that
\[
[X(\bfx,1),Y]=A.
\]
\end{proof}
\begin{rem}
Our approach cannot be modified to yield Thompson's result that $X$
can be taken to be nilpotent. The reason for this is that $X(\bfx,a)$
is nilpotent if and only if $P$ is nilpotent if and only if $a=0$.
Therefore, even if $X(\bfx,a)$ is nilpotent and $\gl_{n}(K)$-regular,
it cannot be $\sl_{n}(K)$-regular, because $\tr(X(\bfx,0)^{r})=0$
for every $r=1,\dots,n-1$. 
\end{rem}

\section{Proof of the Main Theorem}

Throughout this section, $R$ is an arbitrary PID with fraction field
$F$. Note that we consider fields as special types of PIDs.

Before proving our main result (Theorem~\ref{thm:MainPIDs} below),
we give a new and simplified proof of the main result in \cite{commutatorsPID-2016}
that any $A\in\sl_{n}(R)$ is a commutator of matrices in $\gl_{n}(R)$.
The proof of our main result is a bit harder, as it involves a special
analysis for certain prime ideals. Both proofs makek essential use
of the Laffey-Reams form and rely on the following key result:
\begin{lem}
\label{lem:tr(XrA)-is-zero}Suppose that $A=(a_{ij})\in\sl_{n}(R)$
is in Laffey-Reams form, that is, $a_{ij}=0$ for $j\geq i+2$ and
$A\equiv a_{11}1_{n}\mod{(a_{12})}$. Then there exists an ${\bf x}=(x_{1},\dots,x_{n-1})^{\mathsf{{T}}}\in R^{n-1}$,
with $x_{n-1}=a_{11}$, such that 
\[
\tr(X(\bfx,a_{12})^{r}A)=0,
\]
for each $r=1,\dots,n-1$.
\end{lem}
\begin{proof}
By Lemma~\ref{lem:X_n(a,b)-powers} we have 
\[
X(\bfx,a_{12})^{r}=\begin{pmatrix}0 & \overline{0}\\
P^{r-1}\bfx & P^{r}
\end{pmatrix},
\]
where $P=(p_{ij})$, $1\leq i,j,\leq n-1$ is such that $p_{i,i+1}=1$
for $i=1,\dots,n-2$, $p_{n-1,1}=a_{12}$ and $p_{ij}=0$ otherwise
(i.e., $P$ is the row-wise companion matrix of $x^{n-1}-a_{12}$).
Writing $A$ in block-form, we have
\[
A=\begin{pmatrix}a_{11} & (a_{12},0,\dots,0)\\
\bfa & Q
\end{pmatrix},
\]
where $\bfa$ is an $n\times1$ matrix and $Q\in\gl_{n-1}(R)$. Thus
\[
X(\bfx,a_{12})^{r}A=\begin{pmatrix}0 & \overline{0}\\
a_{11}P^{r-1}\bfx+P^{r}\bfa & Q'
\end{pmatrix},
\]
where $Q'=P^{r-1}\bfx(a_{12},0,\dots,0)+P^{r}Q$. Thus, by Lemma~\ref{lem:Pr-1-trace},
\[
\tr(X(\bfx,a_{12})^{r}A)=\tr(Q')=a_{12}x_{r}+\tr(P^{r}Q),
\]
for each $r=1,\dots,n-1$. We have $\tr(P^{r})\equiv0\mod{(a_{12})}$,
for $r=1,\dots,n-1$, and since $A\equiv a_{11}1_{n}\mod{(a_{12})}$
it follows that $Q\equiv a_{11}1_{n-1}\mod{(a_{12})}$. Thus 
\[
\tr(P^{r}Q)\equiv a_{11}\tr(P^{r})\equiv0\mod{(a_{12})},
\]
so there exist $m_{r}\in R$ such that $\tr(P^{r}Q)=a_{12}m_{r}$,
for each $r=1,\dots,n-1$. Put $x_{r}=-m_{r}$, so that 
\[
\tr(X(\bfx,a_{12})^{r}A)=0,
\]
for $r=1,\dots,n-1$. 

Finally, we claim that $\tr(P^{n-1}Q)=-a_{11}a_{12},$ so that 
\[
x_{n-1}=a_{11}.
\]
Indeed, since $P$ is $\gl_{n-1}(R)$-regular with characteristic
polynomial $x^{n-1}-a_{12}$, we have $P^{n-1}=a_{12}1_{n-1}$, so
$\tr(P^{n-1}Q)=a_{12}\tr(Q)=a_{12}(-a_{11})$, as claimed. 
\end{proof}
The following result is essentially \cite[Theorem~6.3]{commutatorsPID-2016},
but the result here is stronger in that it says that $X$ can be taken
in $\sl_{n}(R)$ and such that it is $\gl_{n}(R/\mfp)$-regular mod
any maximal ideal $\mfp$ of $R$.
\begin{thm}
Let $A\in\sl_{n}(R)$ with $n\geq2$. Then there exist matrices $X\in\sl_{n}(R)$
and $Y\in\gl_{n}(R)$ such that $[X,Y]=A$, where $X$ can be chosen
such that $X_{\mfp}$ is $\gl_{n}(R/\mfp)$-regular for every maximal
ideal $\mfp$ of $R$. 
\end{thm}
\begin{proof}
For $n=2$ this is proved separately (see the proof of \cite[Theorem~6.3]{commutatorsPID-2016}).
Assume from now on that $n\geq3$. First, if $A$ is scalar, then
$A\in\sl_{n}(R)$ implies that either $A=0$ or $n=0$ in $R$. The
former case is trivial, while the latter follows from Lemma~\ref{lem:Identity-commutator}.

Assume now that $A$ is not scalar and let $A=(a_{ij})$. After a
possible $\GL_{n}(R)$-conjugation, we can assume that $A$ is in
Laffey\textendash Reams form; see \cite[Theorem~5.6]{commutatorsPID-2016}.
Moreover, we may assume that $(a_{11},a_{12})=(1)$, because if $d$
is a common divisor of $a_{11}$ and $a_{12}$, we can write $A=dA'$
for $A'$ in Laffey\textendash Reams form and if $A'=[X,Y]$ with
$X,Y$ as in the theorem, then $A=[X,dY]$.

By Lemma~\ref{lem:tr(XrA)-is-zero}, there exists an ${\bf x}=(x_{1},\dots,x_{n-1})^{\mathsf{{T}}}\in R^{n-1}$,
with $x_{n-1}=a_{11}$, such that 
\[
\tr(X(\bfx,a_{12})^{r}A)=0,
\]
for each $r=1,\dots,n-1$. Since $x_{n-1}=a_{11}$ and $(a_{11},a_{12})=(1)$,
we have, for every maximal ideal $\mfp$ of $R$, that either $x_{n-1}\notin\mfp$
or $a_{12}\notin\mfp$, and therefore $X_{\mfp}$ is $\gl_{n}(R/\mfp)$-regular,
by Lemma~\ref{lem:X_n(a,b)-regular}. Thus, by \cite[Proposition~3.3]{commutatorsPID-2016},
there exists a $Y\in\gl_{n}(R)$ such that
\[
[X(\bfx,a_{12}),Y]=A.
\]
\end{proof}
We now come to the proof of our main theorem. Just like the proof
of the above theorem, our proof uses Lemma~\ref{lem:tr(XrA)-is-zero},
but since here $X({\bf x},a_{12})_{\mfp}$ cannot in general be $\sl_{n}(R/\mfp)$-regular
for all maximal ideals (cf.~Remark~\ref{rem:not-sl_n-reg}), we
need to treat the exceptional primes separately, and this requires
us to pass to the localisations $R_{\mfp}$, for various prime ideals
$\mfp\in\Spec(R)$. For an element $X\in\gl_{n}(R)$ we will write
$X(\mfp)$ for its canonical image in $\gl_{n}(R_{\mfp})$, not to
be confused with $X_{\mfp}\in\gl_{n}(R/\mfp)$. For any element $x\in R$,
we will use the same symbol $x$ to denote the image of $x$ under
the canonical injection $R\hookrightarrow R_{\mfp}$, and the context
will make it clear in which ring we are working. Similarly, we will
denote the maximal ideal of $R_{\mfp}$ by $\mfp$ and will identify
$X_{\mfp}\in\gl_{n}(R/\mfp)$ with the image of $X(\mfp)$ in $\gl_{n}(R_{\mfp}/\mfp)$.

We will prove that for fixed $A,X\in\sl_{n}(R)$, and for any maximal
ideal $\mfp$ of $R$, there exists a solution $Y(\mfp)\in\sl_{n}(R_{\mfp})$
to the localised equation $[X(\mfp),Y(\mfp)]=A(\mfp)$. Since the
equations $[X,Y]=A$, $\tr(Y)=0$ in $Y$ are equivalent to a system
of linear equations in the entries of $Y$, the well known (and easy
to prove) local-global principle for systems of linear equations (see,
e.g., \cite[Proposition~1]{Hermida}) implies the existence of a global
solution.
\begin{thm}
\label{thm:MainPIDs}Let $A\in\sl_{n}(R)$ for $n\geq3$. Then there
exist matrices $X,Y\in\sl_{n}(R)$ such that $[X,Y]=A$, where $X$
can be chosen such that $X_{\mfp}$ is $\gl_{n}(R/\mfp)$-regular
for every maximal ideal $\mfp$ of $R$. Moreover, $X$ can be chosen
such that $X_{\mfp}$ is $\sl_{n}(R/\mfp)$-regular for every $\mfp$
such that $A_{\mfp}$ is not scalar. 
\end{thm}
\begin{proof}
Assume first that $A$ is scalar. Then $A\in\sl_{n}(R)$ implies that
either $A=0$ or $n=0$ in $R$. The former case is trivial, while
the latter follows from Lemma~\ref{lem:Identity-commutator}.

Assume from now on that $A$ is not scalar and let $A=(a_{ij})$.
After a possible $\GL_{n}(R)$-conjugation, we can assume that $A$
is in Laffey\textendash Reams form. Moreover, we may assume that $(a_{11},a_{12})=(1)$,
because if $d$ is a common divisor of $a_{11}$ and $a_{12}$, we
can write $A=dA'$ for $A'$ in Laffey\textendash Reams form, and
if $A'$ is a commutator of two matrices in $\sl_{n}(R)$, then so
is $A$. 

By Lemma~\ref{lem:tr(XrA)-is-zero}, there exists an ${\bf x}=(x_{1},\dots,x_{n-1})^{\mathsf{{T}}}\in R^{n-1}$,
with $x_{n-1}=a_{11}$, such that 
\[
\tr(X(\bfx,a_{12})^{r}A)=0,
\]
for each $r=1,\dots,n-1$. From now on, let $X:=X(\bfx,a_{12})$.
Since $(a_{11},a_{12})=(1)$, we have, for every maximal ideal $\mfp$
of $R$, that either $x_{n-1}\notin\mfp$ or $a_{12}\notin\mfp$,
and therefore that $X_{\mfp}$ is $\gl_{n}(R/\mfp)$-regular; see
Lemma~\ref{lem:X_n(a,b)-regular}. Moreover, since $A$ is in Laffey-Reams
form, we have $A\equiv a_{11}1_{n}\bmod(a_{12})$, and this, combined
with the fact that $\tr(A)=0$ and $(a_{11},a_{12})=(1)$, implies
that
\begin{equation}
n\in(a_{12}).\label{eq:n-in(a12)}
\end{equation}

We will now pass to the localisations $R_{\mfp}$ for various maximal
ideals $\mfp$ of $R$. Let $\mfp$ be any maximal ideal of $R$.
Then we have the local relation
\[
\tr(X(\mfp)^{r}A(\mfp))=0,\qquad r=1,\dots,n-1.
\]
in $R_{\mfp}$. First, suppose that $A_{\mfp}$ is not scalar. Then
$a_{12}\notin\mfp$, so the matrix $X(\mfp)_{\mfp}=X_{\mfp}$ is $\sl_{n}(R_{\mfp}/\mfp)$-regular,
by Lemma~\ref{lem:X_n(a,b)-regular}, and so, by Proposition~\ref{prop:Criterion},
there exists a $Y(\mfp)\in\sl_{n}(R_{\mfp})$ such that 
\[
[X(\mfp),Y(\mfp)]=A(\mfp).
\]

Next, suppose that $A_{\mfp}$ is scalar, so that $a_{12}\in\mfp$.
Since $a_{12}\neq0$, $X$ is $\sl_{n}(F)$-regular as an element
of $\sl_{n}(F)$, by Lemma~\ref{lem:X_n(a,b)-regular}, so there
exists a $Y(0)\in\sl_{n}(F)$ such that $[X,Y(0)]=A$. Clearing denominators
in $Y(0)$ and passing to the localisation at $\mfp$, we conclude
that there exists a power $p^{m}$ of a generator $p\in R_{\mfp}$
of $\mfp$ and a $Q\in\sl_{n}(R_{\mfp})$, such that 
\begin{equation}
[X(\mfp),Q]=p^{m}A(\mfp).\label{eq:=00005BX,Q=00005D=00003DpmA}
\end{equation}
Let $m\geq0$ be the minimal integer such that (\ref{eq:=00005BX,Q=00005D=00003DpmA})
holds for some $Q\in\sl_{n}(R_{\mfp})$. We will show that $m=0$.
For a contradiction, assume that $m\geq1$. Reducing (\ref{eq:=00005BX,Q=00005D=00003DpmA})
mod $\mfp$, we obtain $[X_{\mfp},Q_{\mfp}]=0$, so $Q_{\mfp}$ commutes
with $X_{\mfp}$. Since $X_{\mfp}$ is $\gl_{n}(R/\mfp)$-regular,
\[
Q=f(X(\mfp))+pD,
\]
for some polynomial $f(T)\in R_{\mfp}[T]$ of degree $n-1$ and some
$D\in\gl_{n}(R_{\mfp})$. Write, $f(T)=c_{0}+c_{1}T+\dots+c_{n-1}T^{n-1}$,
for $c_{i}\in R_{\mfp}$. By Lemma~(\ref{lem:X_n(a,b)-powers}),
we have 
\[
\tr(X^{i})=\begin{cases}
n & \text{if }i=0,\\
(n-1)a_{12} & \text{if }i=n-1,\\
0 & \text{otherwise}
\end{cases}
\]
which implies
\begin{equation}
\tr(X(\mfp)^{i})=\begin{cases}
n & \text{if }i=0,\\
(n-1)a_{12} & \text{if }i=n-1,\\
0 & \text{otherwise}.
\end{cases}\label{eq:trace-X-powers}
\end{equation}
Hence
\begin{equation}
0=\tr(Q)=\sum_{i=0}^{n-1}c_{i}\tr(X(\mfp)^{i})+p\tr(D)=c_{0}n+c_{n-1}(n-1)a_{12}+p\tr(D).\label{eq:tr(Q)-with-tr(D)}
\end{equation}
Moreover, we have $[X(\mfp),Q]=[X(\mfp),pD]=p^{m}A(\mfp)$, so
\[
0=\tr(pDp^{m}A(\mfp))=p^{m+1}\tr(DA(\mfp)),
\]
and thus $\tr(DA(\mfp))=0$. Since $A(\mfp)\equiv a_{11}1_{n}\bmod(a_{12})$
and $(a_{11},a_{12})=(1)$, we conclude that 
\begin{equation}
\tr(D)\in(a_{12}).\label{eq:tr(D)-in(a12)}
\end{equation}
Since $n\in(a_{12})$ by (\ref{eq:n-in(a12)}), we have $n=a_{12}n'$
for some $n'\in R_{\mfp}$. Moreover, since $R_{\mfp}$ is a local
ring, $n-1$ is a unit in $R_{\mfp}$, so we can define the matrix
\[
Q'=(c_{0}n'(n-1)^{-1}+c_{n-1})X^{n-1}+pD.
\]
By (\ref{eq:trace-X-powers}) and (\ref{eq:tr(Q)-with-tr(D)}) we
have
\[
\tr(Q')=c_{0}n+c_{n-1}(n-1)a_{12}+p\tr(D)=\tr(Q)=0,
\]
which, by (\ref{eq:tr(D)-in(a12)}), implies that $c_{0}n+c_{n-1}(n-1)a_{12}\in(pa_{12})$,
and thus 
\[
c_{0}n'(n-1)^{-1}+c_{n-1}\in(p).
\]
Writing $c_{0}n'(n-1)^{-1}+c_{n-1}=p\alpha$ for some $\alpha\in R_{\mfp}$,
we then get
\[
[X(\mfp),Q]=[X(\mfp),pD]=[X(\mfp),Q']=p[X(\mfp),\alpha X^{n-1}+D]=p^{m}A(\mfp),
\]
where $\tr(\alpha X^{n-1}+D)=0$ because 
\[
p\tr(\alpha X^{n-1}+D)=\tr((c_{0}n'(n-1)^{-1}+c_{n-1})X^{n-1}+pD)=\tr(Q')=0.
\]
By cancelling a factor of $p$, we obtain a contradiction to the minimality
of $m$ in (\ref{eq:=00005BX,Q=00005D=00003DpmA}). Thus $m=0$, so
there exists a $Y(\mfp)\in\sl_{n}(R_{\mfp})$ such that $[X(\mfp),Y(\mfp)]=A(\mfp)$.

We have thus proved that for any maximal ideal $\mfp$ of $R$, there
exists a $Y(\mfp)\in\sl_{n}(R_{\mfp})$ such that 
\[
[X(\mfp),Y(\mfp)]=A(\mfp).
\]
Thus, by the local-global principle for systems of linear equations
(see, e.g., \cite[Proposition~1]{Hermida}), there exists a $Y\in\sl_{n}(R)$
such that
\[
[X,Y]=A.
\]
\end{proof}
In the same way as in \cite[Corollary~6.4]{commutatorsPID-2016},
Theorem~\ref{thm:MainPIDs} implies the analogous statement over
any principal ideal ring (PIR), thanks to a theorem of Hungerford
that any PIR is a finite product of homomorphic images of PIDs.

\section{\label{sec:Shalevs-conj}Shalev's conjecture for $n=2$}

This section is devoted to a proof of Shalev's conjecture mentioned
in the introduction, in the case $n=2$. For a group $G$ and elements
$x,y\in G$ we write the commutator as $(x,y)=xyx^{-1}y^{-1}$. In
this section, $R$ will denote a local PID (i.e., a discrete valuation
ring) with residue field $k$. We denote the maximal ideal in $R$
by $\mfp$ and let $\pi\in\mfp$ be a generator. We first prove a
result whose conclusion is weaker than Shalev's conjecture for $n=2$,
in that one of the elements is only shown to lie in $\GL_{2}(R)$,
but where the hypotheses are slightly more general in that we allow
any residue field apart from $\F_{2}$. We will then refine this result
to prove Proposition~\ref{prop:Shalevs-conj}, which contains Shalev's
conjecture for $n=2$ as a special case.

Recall that an element $X\in\gl_{n}(R)$ is $\gl_{n}(R)$-regular
if and only if $X_{\mfp}$ is $\gl_{n}(k)$-regular; see \cite[Lemma~2.5]{commutatorsPID-2016}.
Moreover, if $X$ is $\gl_{n}(R)$-regular then $X$ is $\GL_{n}(R)$-conjugate
to a companion matrix; see \cite[Lemma~2.3]{commutatorsPID-2016}.
By convention, we will write companion matrices row-wise, that is,
with ones on the super-diagonal. If $X\in\gl_{n}(R)$ has units on
the superdiagonal and zeros above the superdiagonal, then it is regular;
see \cite[Lemma~2.7]{commutatorsPID-2016} (note that to go from units
to ones on the superdiagonal, one only needs to conjugate by a diagonal
element). The same conclusion holds if `superdiagonal' is replaced
by `subdiagonal'.
\begin{prop}
\label{prop:ShalevGL2}Let $R$ be a local PID such that $|k|>2$.
Then, for every $A\in\SL_{2}(R)$ there exist $x\in\GL_{2}(R)$ and
$y\in\SL_{2}(R)$ such that $(x,y)=A$. Moreover, $y$ can be taken
to be $\GL_{2}(R)$-conjugate to an element of the form $\left(\begin{smallmatrix}0 & 1\\
-1 & s
\end{smallmatrix}\right)\in\SL_{2}(R)$.
\end{prop}
\begin{proof}
The strategy of the proof is the following. We find a $y\in\SL_{2}(R)$
such that both $y$ and $yA$ are regular as elements in $\gl_{n}(R)$,
and such that $y$ and $yA$ have the same determinant and trace.
This implies that $y$ and $yA$ are $\GL_{2}(R)$-conjugate, that
is, there exists an $x\in\GL_{2}(R)$ such that 
\[
yA=xyx^{-1},
\]
and so $(x,y)=A$. 

By \cite[Lemma~2.7]{commutatorsPID-2016}, any 
\[
y=\begin{pmatrix}y_{11} & 1\\
y_{21} & y_{22}
\end{pmatrix}\in\gl_{n}(R)
\]
is $\gl_{n}(R)$-regular. In order for $\det(y)=\det(yA)$ we need
$\det(y)=1$, so $y_{21}=y_{11}y_{22}-1$, which we assume henceforth.
We distinguish two cases: either the image $A_{\mfp}$ of $A$ in
$\SL_{2}(k)$ is scalar or regular.

\smallskip\noindent\textbf{Case 1:} Assume that $A_{\mfp}$ is scalar.
If moreover $A$ is scalar, then $A=\pm1$, so $yA=\pm y$ is regular,
and $\tr(yA)=\pm\tr(y)$. Setting $y_{22}=-y_{11}$ so that $\tr(y)=0$,
we obtain the existence of an $x\in\GL_{2}(R)$ such that $yA=xyx^{-1}$.

If $A_{\mfp}$ is scalar but $A$ is not scalar, then up to conjugation
of $A$, we can write 
\[
A=\lambda1+\pi^{i}A',
\]
where $\lambda=\pm1$, $i\geq1$ and $A'=\begin{pmatrix}0 & 1\\
a' & b'
\end{pmatrix}\in\SL_{2}(R)$. Then, since $y_{\mfp}\in\SL_{2}(k)$ is regular, and 
\[
(yA)_{\mfp}=\pm y_{\mfp},
\]
we also have that $(yA)_{\mfp}$ is regular, and hence (because $R$
is local), that $yA$ is regular. Furthermore, we need to ensure that
\begin{equation}
\tr(yA)=y_{11}(\pi^{i}y_{22}+\lambda)+y_{22}(b'\pi^{i}+\lambda)+\pi^{i}(a'-1)=\tr(y)=y_{11}+y_{22}\label{eq:trace-scalarmodp}
\end{equation}
If $\chara k\neq2$ and $\lambda=-1$, then $\lambda-1$ is a unit,
so this equation clearly has a solution $y_{11}\in R$ for any $y_{22}\in R$.
If $\lambda=1$ (\ref{eq:trace-scalarmodp}) is equivalent to 
\[
y_{11}y_{22}+y_{22}b'+a'-1=0,
\]
which has a solution over $R$, for example $y_{11}=-a'-b'+1,y_{22}=1$.

\smallskip\noindent\textbf{Case 2:} Assume that $A_{\mfp}$ is non-scalar.
Then $A_{\mfp}$ is regular, so $A$ is regular, and hence, up to
conjugation, we can write
\[
A=\begin{pmatrix}0 & 1\\
-1 & b
\end{pmatrix},
\]
for $b\in R$. For $y=\begin{pmatrix}y_{11} & 1\\
y_{11}y_{22}-1 & y_{22}
\end{pmatrix}$, we have 
\[
yA=\begin{pmatrix}-1 & y_{11}+b\\
-y_{22} & y_{22}(y_{11}+b)-1
\end{pmatrix}.
\]
Since the residue field of $R$ has at least three elements, we can
choose $y_{22}$ such that both $y_{22}$ and $y_{22}-1$ are units.
Then $yA$ is regular by \cite[Lemma~2.7]{commutatorsPID-2016}, and
the equation 
\[
\tr(yA)=y_{22}(y_{11}+b)-2=\tr(y)=y_{11}+y_{22},
\]
has a solution $y_{11}\in R$. Hence, there exists an $x\in\GL_{2}(R)$
such that $ya=xyx^{-1}$.

Finally, $y=\begin{pmatrix}y_{11} & 1\\
y_{11}y_{22}-1 & y_{22}
\end{pmatrix}$ is regular, so it is $\GL_{2}(R)$-conjugate to $y=\begin{pmatrix}0 & 1\\
-1 & y_{11}+y_{22}
\end{pmatrix}$.
\end{proof}
Note that the hypothesis on the residue field in the above proposition
is optimal, since $\GL_{2}(\F_{2})=\SL_{2}(\F_{2})$, and this group
is not perfect.

The following result implies Shalev's conjecture mentioned in the
introduction, in the case $n=2$.
\begin{prop}
\label{prop:Shalevs-conj}Let $R$ be a Henselian discrete valuation
ring with finite residue field $k$ such that $\chara k>2$ and $|k|>3$.
Then, for every $A\in\SL_{2}(R)$ there exist $x\in\SL_{2}(R)$ and
$y\in\SL_{2}(R)$ such that $(x,y)=A$. 
\end{prop}
\begin{proof}
If $A_{\mfp}$ is non-scalar, the result follows from \cite[Theorem~3.5]{AGKS-SLnZp}.
Assume therefore that $A_{\mfp}$ is scalar. By Proposition~\ref{prop:ShalevGL2}
we have $(x,y)=a$ for some $x\in\GL_{2}(R)$ and some $y$ conjugate
to $\left(\begin{smallmatrix}0 & 1\\
-1 & s
\end{smallmatrix}\right)$, for $s\in R$. As before, let $y_{\mfp}\in\SL_{2}(k)$ denote the
image of $y$ under the canonical map $\SL_{2}(R)\rightarrow\SL_{2}(k)$.
We will show that $s$ can be chosen such that $y_{\mfp}$ is semisimple.
If $y_{\mfp}$ is not semisimple, then it has characteristic polynomial
$x^{2}-sx+1\equiv(x\pm1)^{2}\bmod\mfp$, so $s\equiv\pm2\bmod\mfp$.
We therefore want to show that $s$ can be chosen such that $s\not\equiv\pm2\bmod\mfp$.
We will examine Case~1 of the proof of Proposition~\ref{prop:ShalevGL2}
to show that this can indeed be achieved. As before, write

\[
A=\lambda1+\pi^{i}A',
\]
where $\lambda=\pm1$ and $\begin{pmatrix}0 & 1\\
a' & b'
\end{pmatrix}$. We have seen that if $A$ is scalar, we can choose $s=0\not\equiv\pm2\bmod\mfp$,
and if $\lambda=-1$, (\ref{eq:trace-scalarmodp}) implies that $s=0\not\equiv\pm2\bmod\mfp$.
Moreover, if $\lambda=1$, the last part of Case~1 of the proof of
Proposition~\ref{prop:ShalevGL2} says that we need to consider the
relation 
\begin{equation}
y_{11}y_{22}+y_{22}b'+a'-1=0.\label{eq:y11-y22-a-b}
\end{equation}
Note that $\det(A)=1+b'\pi-a'\pi^{2}=1$, so $b'\in\mfp$. We consider
the number of solutions of (\ref{eq:y11-y22-a-b}) mod $\mfp$ such
that $s=y_{11}+y_{22}\equiv\pm2\bmod\mfp$. The system of congruences
\[
\begin{cases}
(y_{11}y_{22}+a'-1)_{\mfp} & =0,\\
(y_{11}+y_{22})_{\mfp} & =\pm2.
\end{cases}
\]
has at most $4$ distinct solutions $((y_{11})_{\mfp},(y_{22})_{\mfp})\in k^{2}$
(at most two for each choice of sign for $\pm2$). On the other hand,
the equation
\[
(y_{11}y_{22}+a'-1)_{\mfp}=0
\]
has at least $|k|-1$ distinct solutions, so when $|k|>5$, there
exist $(y_{11})_{\mfp},(y_{22})_{\mfp}\in k$ such that $(y_{11}+y_{22})_{\mfp}\neq\pm2$.
Furthermore, if $k=\F_{5}$, the equation $(y_{11}y_{22}+a'-1)_{\mfp}=0$
has $9$ solutions if $a'_{\mfp}=1$, so it only remains to consider
the case $k=\F_{5}$, with $a'_{\mfp}\neq1$. In this case, one checks
easily that either $(y_{11})_{\mfp}=1,(y_{22})_{\mfp}=1-a'$ or $(y_{11})_{\mfp}=2,(y_{22})_{\mfp}=(1-a')2^{-1}$
is a solution to $(y_{11}y_{22}+a'-1)_{\mfp}=0$ such that $(y_{11}+y_{22})_{\mfp}\neq\pm2$.
We have therefore shown that whenever $|k|>3$, there exist $y_{11},y_{22}\in R$
such that $y_{11}y_{22}+a'-1\equiv0\bmod\mfp$ and such that $s=y_{11}+y_{22}\not\equiv\pm2\bmod\mfp$.
Hensel's lemma now implies that equation (\ref{eq:y11-y22-a-b}) has
a solution $y_{11},y_{22}\in R$ such that $s=y_{11}+y_{22}\not\equiv\pm2\bmod\mfp$.
We thus conclude that $s=y_{11}+y_{22}$ can be chosen such that $y_{\mfp}$
is semisimple.

Suppose, as we may, that $s\not\equiv\pm2\bmod\mfp$, so that $y_{\mfp}$
is semisimple. We claim that the determinant map
\[
\det:C_{\GL_{2}(k)}(y_{\mfp})\longrightarrow k^{\times}
\]
 is surjective. Indeed, since $y_{\mfp}$ is regular, we have
\[
C_{\GL_{2}(k)}(y_{\mfp})=k[y_{\mfp}]^{\times},
\]
and if the characteristic polynomial of $y_{\mfp}$ is irreducible,
$k[y_{\mfp}]$ is a field, while otherwise $y_{\mfp}$ has distinct
eigenvalues in $k$. When $k[y_{\mfp}]$ is a field, the determinant
coincides with the norm map $N:k[y_{\mfp}]^{\times}\rightarrow k^{\times}$,
which is well-known to be surjective (because $k$ is finite). When
$k[y_{\mfp}]$ has distinct eigenvalues in $k$, the centraliser $C_{\GL_{2}(k)}(k[y_{\mfp}])$
is $\GL_{2}(k)$-conjugate to the diagonal subgroup of $\GL_{2}(k)$,
on which $\det$ is clearly surjective. We now want to show that 
\[
\det:C_{\GL_{2}(R)}(y)\longrightarrow R^{\times}
\]
is surjective. Since $y$ is regular, we have
\[
C_{\gl_{2}(R)}(y)=R[y]=\left\{ \begin{pmatrix}\alpha & \beta\\
-\beta & \alpha+\beta s
\end{pmatrix}\mid\alpha,\beta\in R,\right\} ,
\]
so the determinant is surjective if $\alpha^{2}+\alpha\beta s+\beta^{2}=r$
has a solution $\alpha,\beta\in R$ for each $r\in R^{\times}$. We
show that Hensel's lemma implies that this equation has a solution
$\alpha,\beta\in R$. Indeed, if the gradient
\[
\begin{pmatrix}2\alpha+\beta s\\
2\beta+\alpha s
\end{pmatrix}\equiv0\mod\mfp,
\]
then $\alpha\equiv\frac{\alpha s^{2}}{4}$, so either $\alpha\equiv0$
or $s^{2}\equiv4$. In the same way, either $\beta\equiv0$ or $s^{2}\equiv4$.
Then, since we have assumed that $s\not\equiv\pm2$, we obtain, $\alpha\equiv\beta\equiv0$,
which is not the case if $\alpha^{2}+\alpha\beta s+\beta^{2}=r\in R^{\times}$.
Thus, any solution to this equation mod $\mfp$ has a lift to $R$.
Since we have just observed that this equation has a solution mod
$\mfp$ for every $r\in R^{\times}$, we conclude that $\det:C_{\GL_{2}(R)}(y)\rightarrow R^{\times}$
is surjective.

We have shown that there exists a $g\in C_{\GL_{2}(R)}(y)$ such that
$\det(g)=\det(x)^{-1}$. Thus 
\[
yA=xyx^{-1}=(xg)y(xg)^{-1},
\]
so $(xg,y)=A$ with $xg\in\SL_{2}(R)$. 
\end{proof}

\bibliographystyle{alex}
\bibliography{alex}

\end{document}